\documentclass[submission,copyright,creativecommons]{eptcs}
\providecommand{\event}{ACT 2019} % Name of the event you are submitting
\usepackage{underscore} % Only needed if you use pdflatex.

\title{Internal lenses as functors and cofunctors}
\author{Bryce Clarke\thanks{The author is supported by the 
Australian Government Research Training Program Scholarship.}
\institute{Centre of Australian Category Theory\\ 
Macquarie University, Australia}
\email{bryce.clarke1@hdr.mq.edu.au}
}

\def\titlerunning{Internal lenses as functors and cofunctors}
\def\authorrunning{Bryce Clarke}

\usepackage{mathtools}
\usepackage{bbold}
\usepackage{amsthm,amssymb}
\usepackage{tikz}
\usepackage{tikz-cd}
\usepackage{lipsum}
\usepackage{faktor}

\newtheorem{theorem}{Theorem}

\newtheorem{prop}[theorem]{Proposition}
\newtheorem{cor}[theorem]{Corollary}

\theoremstyle{definition}
\newtheorem{defn}[theorem]{Definition}
\newtheorem{example}[theorem]{Example}

\theoremstyle{remark}
\newtheorem*{rem}{Remark}
\newtheorem*{notation}{Notation}

\newcommand{\E}{\mathcal{E}}
\newcommand{\Arr}{\Phi}
\newcommand{\br}{\rightleftharpoons}
\newcommand{\phibar}{\overline{\varphi}}
\newcommand{\comma}[2]{{#1}\downarrow{#2}}

\newcommand{\Set}{\mathbf{Set}}
\newcommand{\Cat}{\mathbf{Cat}}
\newcommand{\Dbl}{\mathbf{Dbl}}
\newcommand{\Kat}{\mathsf{Cat}}
\newcommand{\Cof}{\mathsf{Cof}}
\newcommand{\Lens}{\mathsf{Lens}}
\newcommand{\DOpf}{\mathsf{DOpf}}

\begin{document}
\maketitle

\begin{abstract}
Lenses may be characterised as objects in the category of algebras over
a monad, however they are often understood instead as morphisms, which 
propagate updates between systems. 
Working internally to a category with pullbacks, we define lenses as 
simultaneously functors and cofunctors between categories.
We show that lenses may be canonically represented as a particular 
commuting triangle of functors, and unify the classical state-based 
lenses with both c-lenses and d-lenses in this framework. 
This new treatment of lenses leads to considerable simplifications that 
are important in applications, including a clear interpretation of lens
composition.
\end{abstract}

%%%%%%%%%%%%%%%%%%%%%%%%%%%%%%%%%%%%%%%%%%%%%%%%%%%%%%%%%%%%%%%%%%%%%%%%
\section{Introduction}\label{S:introduction}
%%%%%%%%%%%%%%%%%%%%%%%%%%%%%%%%%%%%%%%%%%%%%%%%%%%%%%%%%%%%%%%%%%%%%%%%

Lenses form a mathematical structure that aims to capture the 
fundamental aspects of certain synchronisations between pairs of 
systems.
The central goal of such synchronisation is to coherently propagate 
updates in one system to updates in another, and vice versa. 
The precise nature of the synchronisation process depends closely on the 
type of system being studied, and thus many different kinds of lenses
have been defined to characterise various applications and examples.

Although a relatively recent subject for detailed abstract study, lenses 
are an impressive example of applied category theory, playing major 
roles in database view updating, in Haskell programs of many kinds, and
in diverse examples of Systems Interoperations, Data Sharing, and 
Model-Driven Engineering.  
Thus, further clarifying the category-theoretic status and systematising
the use of lenses, as this paper aims to do, is an important part of 
applied category theory.

Lenses were originally introduced \cite{FGMPS07} to provide a solution 
to the view-update problem \cite{BS81}.  
In treatments of the view-update problem systems are generally modelled
as a set of states, where it is possible to update from one state of the
system to any other, and the only information retained about this update
are its initial and final states. 
Thus a system may be understood as a \emph{codiscrete category} on its 
set of states $A$ with set of updates $A \times A$ given by a pair of
initial and final states.

Lenses have long been recognised to be some kind of morphism between 
systems. 
An obvious notion of morphism between systems is simply a function 
$f \colon A \rightarrow B$ between their sets of states. 
Since systems may be modelled as codiscrete categories, there is also an
induced function $f \times f \colon A \times A \rightarrow B \times B$
between the sets of updates of these systems. 
The map $f \colon A \rightarrow B$ is called the \textsf{Get} function
and provides the first component of a lens between the systems $A$ and
$B$, often called the \emph{source} and \emph{view}.

The second component of a lens is called the \textsf{Put} function 
$p \colon A \times B \rightarrow A$ whose role is less obvious. 
The set $A \times B$ may be interpreted as the set of 
\emph{anchored view updates} via the induced function 
$f \times 1_{B} \colon A \times B \rightarrow B \times B$
which produces a view update whose initial state is given by the 
\textsf{Get} function. 
The induced function
$\langle \pi_{0}, p \rangle \colon A \times B \rightarrow A \times A$ 
may be regarded as the \textsf{Put} function, propagating every anchored
view update to a source update, illustrated in the diagram below. 
\begin{equation*}
\begin{tikzcd}
A 
\arrow[d, "f"']
&[-1.5em]
a 
\arrow[d, phantom, "\vdots"]
\arrow[r, dashed]
& 
p(a, b) 
\arrow[d, phantom, "\vdots"]
\\
B
& fa 
\arrow[r]
& b 
\end{tikzcd}
\end{equation*}
 
Frequently the \textsf{Get} and \textsf{Put} functions of a lens are 
required to satisfy three additional axioms, called the 
\emph{lens laws}, which ensure the synchronisation of updates between 
systems is well-behaved. 
\begin{equation*}
\begin{tikzcd}
A \times B 
\arrow[r, "p"]
\arrow[rd, "\pi_{1}"']
& A 
\arrow[d, "f"]
\\
& B
\end{tikzcd}
\qquad \qquad
\begin{tikzcd}
A
\arrow[r, "{\langle 1_{A}, f \rangle}"]
\arrow[rd, "1_{A}"']
& A \times B  
\arrow[d, "p"]
\\
& A
\end{tikzcd}
\qquad \qquad
\begin{tikzcd}
A \times B \times B
\arrow[d, "\pi_{0,2}"']
\arrow[r, "p \times 1_{B}"]
& A \times B 
\arrow[d, "p"]
\\
A \times B
\arrow[r, "p"']
& A
\end{tikzcd}
\end{equation*}
In order from left to right: the \textsf{Put-Get} law ensures that the
systems $A$ and $B$ are indeed synchronised under the \textsf{Get} and
\textsf{Put} functions; the \textsf{Get-Put} law ensures that anchored 
view updates which are identities are preserved by the \textsf{Put} 
function; the \textsf{Put-Put} law ensures that composite anchored view 
updates are preserved under the \textsf{Put} function. 

In summary, a \emph{state-based lens} \cite{FGMPS07}, denoted
$(f, p) \colon A \br B$, consists of a \textsf{Get} function 
$f \colon A \rightarrow B$ and a \textsf{Put} function 
$p \colon A \times B \rightarrow A$ satisfying the lens laws. 
Early mathematical work \cite{JRW10} characterised state-based lenses as
algebras for a well-known monad,
\begin{align*}
\faktor{\Set}{B} \quad &\longrightarrow \quad \faktor{\Set}{B} \\
f \colon A \rightarrow B \quad &\longmapsto \quad 
\pi_{1} \colon A \times B \rightarrow B
\end{align*}
which may be generalised to any category with finite products. 
It was later shown that lenses are also coalgebras for a comonad 
\cite{GJ12} and may be defined inside any cartesian closed category. 
While these works took the first steps towards internalisation of 
lenses, they characterised lenses as objects in the category of
Eilenberg-Moore (co)algebras, rather than morphisms between sets, and 
did not account for composition of lenses. 

A significant shortcoming of state-based lenses in many applications is
they only describe synchronisation between systems as a set of states, 
or codiscrete categories, ignoring the information on how states are 
updated.
This motivated the independent development of both \emph{c-lenses} 
\cite{JRW12} and \emph{d-lenses} \cite{DXC11} between systems modelled 
as arbitrary categories. 
Making use of comma categories instead of products, c-lenses were 
defined as algebras for a classical KZ-monad \cite{Str74}, and may be 
also understood as split Grothendieck opfibrations. 
In contrast d-lenses were shown \cite{JR13} to be more general, as split
opfibrations \emph{without} the usual universal property, and could only
be characterised as algebras for a \emph{semi-monad} satisfying an 
additional axiom. 

Later work \cite{JR16} showed that the category of state-based lenses 
(as morphisms) is a full subcategory of the category of d-lenses 
(which also contains a subcategory of c-lenses).
Despite this unification of category-based lenses, composition was still
defined in an ad hoc fashion, and there was no mathematical explanation 
as to why lenses characterised as algebras should be understood as 
morphisms. 

\subsection*{Summary of Paper}

The contribution of this paper may be summarised as follows: 
\begin{itemize}
\item Generalise the theory of lenses to be internal to any category 
$\E$ with pullbacks.
\item Define an \emph{internal lens} as an internal functor and an 
internal cofunctor, which provide the appropriate notion of \textsf{Get}
and \textsf{Put}, respectively. 
\item Characterise internal lenses as diagrams of internal functors,
using the span representation of an internal cofunctor. 
\item Show there is a well-defined category $\Lens(\E)$ whose objects 
are internal categories and whose morphisms are internal lenses. 
\item Demonstrate state-based lenses, c-lenses, and d-lenses as examples 
of internal lenses. 
\end{itemize}

%%%%%%%%%%%%%%%%%%%%%%%%%%%%%%%%%%%%%%%%%%%%%%%%%%%%%%%%%%%%%%%%%%%%%%%%
\section{Background}\label{S:background}
%%%%%%%%%%%%%%%%%%%%%%%%%%%%%%%%%%%%%%%%%%%%%%%%%%%%%%%%%%%%%%%%%%%%%%%%

This section provides a brief review of the relevant internal category
theory required for the paper, most of which can be found in standard 
references such as \cite{Bor94, Joh02, Mac98}. 
Throughout we work internal to a category $\E$ with pullbacks, with the
main examples being $\E = \Set, \Cat$. 

The idea is that a system may be defined as an internal category with an 
\emph{object of states} and an \emph{object of updates}. 
An internal functor will later be interpreted as the \textsf{Get} 
component of an internal lens, while internal discrete opfibrations will
also be central in defining the \textsf{Put} component of an internal 
lens.
Codiscrete categories and arrow categories are presented as examples and 
will later be used to define internal versions of state-based lenses and
c-lenses. 

\begin{defn}\label{defn:internalcat}
An \emph{internal category} $A$ consists of an \emph{object of objects}
$A_{0}$ and an \emph{object of morphisms} $A_{1}$ together with a span,
\begin{equation}
\begin{tikzcd}[row sep = small, column sep = small]
& A_{1}
\arrow[ld, "d_{1}"']
\arrow[rd, "d_{0}"]
& \\
A_{0}
& & A_{0}
\end{tikzcd}
\end{equation}
where $d_{1} \colon A_{1} \rightarrow A_{0}$ is the \emph{domain map} 
and $d_{0} \colon A_{1} \rightarrow A_{0}$ is the \emph{codomain map},
and the pullbacks,
\begin{equation}\label{dgrm:internalpb}
\begin{tikzcd}[row sep = small, column sep = small]
& A_{2}
\arrow[ld, "d_{2}"']
\arrow[rd, "d_{0}"]
& \\
A_{1}
& & A_{1} \\
& A_{0}
\arrow[from=ru, "d_{1}"]
\arrow[from=lu, "d_{0}"']
\arrow[from=uu, phantom, "\lrcorner" rotate = -45, very near start]
&
\end{tikzcd}
\qquad \qquad
\begin{tikzcd}[row sep = small, column sep = small]
& A_{3}
\arrow[ld, "d_{3}"']
\arrow[rd, "d_{0}"]
& \\
A_{2}
& & A_{2} \\
& A_{1}
\arrow[from=ru, "d_{2}"]
\arrow[from=lu, "d_{0}"']
\arrow[from=uu, phantom, "\lrcorner" rotate = -45, very near start]
&
\end{tikzcd}
\end{equation}
where $A_{2}$ is the \emph{object of composable pairs} and $A_{3}$ is 
the \emph{object of composable triples}, as well as an 
\emph{identity map} $i_{0} \colon A_{0} \rightarrow A_{1}$ and 
\emph{composition map} $d_{1} \colon A_{2} \rightarrow A_{1}$ satisfying
the following commutative diagrams: 
\begin{equation}\label{dgrm:internalcat}
\begin{tikzcd}
A_{0}
\arrow[d, "i_{0}"']
\arrow[r, "i_{0}"]
\arrow[rd, "1_{A_{0}}" description]
& A_{1}
\arrow[d, "d_{1}"]
\\
A_{1}
\arrow[r, "d_{0}"']
& A_{0}
\end{tikzcd}
\quad \quad
\begin{tikzcd}
A_{1}
\arrow[d, "d_{1}"']
& A_{2}
\arrow[l, "d_{2}"']
\arrow[r, "d_{0}"]
\arrow[d, "d_{1}"]
& A_{1}
\arrow[d, "d_{0}"]
\\
A_{0}
& A_{1}
\arrow[l, "d_{1}"]
\arrow[r, "d_{0}"']
& A_{0}
\end{tikzcd}
\quad \quad
\begin{tikzcd}
A_{1}
\arrow[d, "i_{1}"']
\arrow[r, "i_{0}"]
\arrow[rd, "1_{A_{1}}" description]
& A_{2}
\arrow[d, "d_{1}"]
\\
A_{2}
\arrow[r, "d_{1}"']
& A_{1}
\end{tikzcd}
\quad \quad
\begin{tikzcd}
A_{3}
\arrow[d, "d_{2}"']
\arrow[r, "d_{1}"]
& A_{2}
\arrow[d, "d_{1}"]
\\
A_{2}
\arrow[r, "d_{1}"']
& A_{1}
\end{tikzcd}
\end{equation}
% Added new line here to replace notation.
The morphisms $i_{0}, i_{1} \colon A_{1} \rightarrow A_{2}$ and 
$d_{1}, d_{2} \colon A_{3} \rightarrow A_{2}$ appearing in 
\eqref{dgrm:internalcat} are defined using the universal property of the
pullback $A_{2}$.
\end{defn}

\begin{example}
A \emph{small category} is an internal category in $\Set$. 
Thus a small category consists of a \emph{set} of objects and a
\emph{set} of morphisms, together with \emph{functions} specifying the
domain, codomain, identity, and composition.
\end{example}

\begin{example}
A (small) \emph{double category} is an internal category in $\Cat$, 
the category of small categories and functors. 
Thus a double category consists of a \emph{category} of objects and 
a \emph{category} of morphisms, together with \emph{functors} specifying
the domain, codomain, identity, and composition.
\end{example}

\begin{example}
Assume $\E$ has finite limits. 
A \emph{codiscrete category} on an object $A \in \E$ is an internal 
category whose object of objects is $A$ and whose object of morphisms is
the product $A \times A$, with domain and codomain maps given by the 
left and right projections:
\begin{equation*}
\begin{tikzcd}[row sep = small, column sep = tiny]
& A \times A
\arrow[ld, "\pi_{0}"']
\arrow[rd, "\pi_{1}"]
& \\
A
& & A
\end{tikzcd}
\end{equation*}
The identity map is given by the diagonal 
$\langle 1_{A}, 1_{A} \rangle \colon A \rightarrow A \times A$,
the object of composable pairs is given by the product 
$A \times A \times A$, and the composition map is given by the following
universal morphism: 
\begin{equation*}
\begin{tikzcd}
& A \times A \times A
\arrow[ld, "\pi_{0}"', start anchor = south west]
\arrow[rd, "\pi_{2}", start anchor = south east]
\arrow[d, "\pi_{0,2}", dashed]
&
\\
A
& A \times A
\arrow[l, "\pi_{0}"]
\arrow[r, "\pi_{1}"']
& A
\end{tikzcd}
\end{equation*}
\end{example}

\begin{example}\label{ex:arrowcat}
Let $A$ be an internal category. 
The \emph{arrow category} $\Arr{A}$ has an object of objects 
$A_{1}$ and an object of morphisms 
$A_{11} \coloneqq A_{2} \times_{A_{1}} A_{2}$ defined by the pullback,
\begin{equation*}
\begin{tikzcd}[column sep = small, row sep = small]
& & A_{11}
\arrow[ld, "\pi_{0}"']
\arrow[rd, "\pi_{1}"]
& & \\
& A_{2}
\arrow[ld, "d_{2}"']
\arrow[rd, "d_{1}"]
& & A_{2}
\arrow[ld, "d_{1}"']
\arrow[rd, "d_{0}"]
& \\
A_{1}
& & A_{1}
\arrow[from=uu, phantom, "\lrcorner" rotate = -45, very near start]
& & A_{1}
\end{tikzcd}
\end{equation*}
with domain map $d_{2}\pi_{0} \colon A_{11} \rightarrow A_{1}$ and 
codomain map $d_{0}\pi_{1} \colon A_{11} \rightarrow A_{1}$. 
The pullback $A_{11}$ may be understood as the 
\emph{object of commutative squares} in $A$. 
The identity and composition maps require tedious notation to define
precisely, however we note they are induced from the diagrams
\eqref{dgrm:internalcat}. 
\end{example}

\begin{defn}\label{defn:internalfunctor}
Let $A$ and $B$ be internal categories. 
An \emph{internal functor} $f \colon A \rightarrow B$ consists of morphisms,
\[
	f_{0} \colon A_{0} \longrightarrow B_{0} 
	\qquad \qquad
	f_{1} \colon A_{1} \longrightarrow B_{1}
\]
satisfying the following commutative diagrams:
\begin{equation}\label{dgrm:internalfunctor}
\begin{tikzcd}
A_{0}
\arrow[d, "f_{0}"']
& A_{1}
\arrow[l, "d_{1}"']
\arrow[r, "d_{0}"]
\arrow[d, "f_{1}"]
& A_{0}
\arrow[d, "f_{0}"]
\\
B_{0}
& B_{1}
\arrow[l, "d_{1}"]
\arrow[r, "d_{0}"']
& B_{0}
\end{tikzcd}
\qquad
\begin{tikzcd}
A_{0}
\arrow[r, "i_{0}"]
\arrow[d, "f_{0}"']
& A_{1}
\arrow[d, "f_{1}"]
\\
B_{0}
\arrow[r, "i_{0}"']
& B_{1}
\end{tikzcd}
\qquad
\begin{tikzcd}
A_{2}
\arrow[r, "d_{1}"]
\arrow[d, "f_{2}"']
& A_{1}
\arrow[d, "f_{1}"]
\\
B_{2}
\arrow[r, "d_{1}"']
& B_{1}
\end{tikzcd}
\end{equation}
% Added new line here to replace notation.
The morphism $f_{2} \colon A_{2} \rightarrow B_{2}$ appearing in 
\eqref{dgrm:internalfunctor} is defined using the universal property of 
the pullback $B_{2}$. 
\end{defn}

\begin{rem}
Given an internal category $A$, the \emph{identity functor} consists of 
a pair of morphisms:
\[
	1_{A_{0}} \colon A_{0} \longrightarrow A_{0}
	\qquad \qquad
	1_{A_{1}} \colon A_{1} \longrightarrow A_{1}
\]
Given internal functors $f \colon A \rightarrow B$ and 
$g \colon B \rightarrow C$, their \emph{composite functor} 
$g \circ f \colon A \rightarrow C$ consists of a pair of morphisms:
\[
	g_{0}f_{0} \colon A_{0} \longrightarrow C_{0}
	\qquad \qquad
	g_{1}f_{1} \colon A_{1} \longrightarrow C_{1}
\]
Composition of internal functors is both unital and associative, 
as it is induced by composition of morphisms in $\E$.
\end{rem}

\begin{defn}
Let $\Kat(\E)$ be the category whose objects are internal categories and 
whose morphisms are internal functors. 
\end{defn}

\begin{example}
The category of sets and functions $\Set$ has pullbacks, thus we obtain
the familiar example $\Cat = \Kat(\Set)$ of small categories and 
functors between them. 
\end{example}

\begin{example}
The category $\Cat$ has pullbacks, so we obtain the category 
$\Dbl = \Kat(\Cat)$ of double categories and double functors between 
them. 
\end{example}

\begin{rem}
The category $\Kat(\E)$ has all pullbacks. Given internal functors 
$f \colon A \rightarrow B$ and $g \colon C \rightarrow B$, their 
pullback is the category $A \times_{B} C$ constructed from the 
pullbacks,
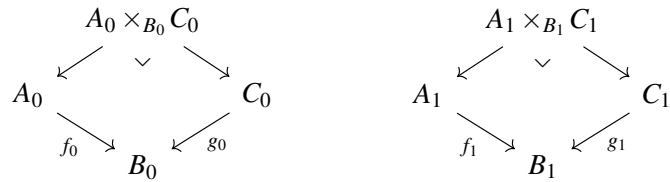
\begin{equation*}
\begin{tikzcd}[row sep = small, column sep = tiny]
& A_{0} \times_{B_{0}} C_{0}
\arrow[ld]
\arrow[rd]
& \\
A_{0}
& & C_{0} \\
& B_{0}
\arrow[from=lu, "f_{0}"']
\arrow[from=ru, "g_{0}"]
\arrow[from=uu, phantom, "\lrcorner" rotate = -45, very near start]
&
\end{tikzcd}
\qquad \qquad 
\begin{tikzcd}[row sep = small, column sep = tiny]
& A_{1} \times_{B_{1}} C_{1}
\arrow[ld]
\arrow[rd]
& \\
A_{1}
& & C_{1} \\
& B_{1}
\arrow[from=lu, "f_{1}"']
\arrow[from=ru, "g_{1}"]
\arrow[from=uu, phantom, "\lrcorner" rotate = -45, very near start]
&
\end{tikzcd}
\end{equation*}
which define the object of objects and object of morphisms, 
respectively. 
The rest of the structure is defined using the universal property
of the pullback. 
Therefore \emph{internal double categories} may be defined as categories
internal to $\Kat(\E)$. 
\end{rem}

\begin{example}
An \emph{internal discrete opfibration} 
is an internal functor $f \colon A \rightarrow B$ 
such that the following diagram is a pullback: 
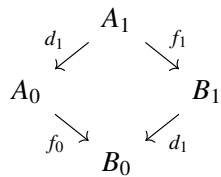
\begin{equation*}
\begin{tikzcd}[row sep = small, column sep = small]
& A_{1}
\arrow[ld, "d_{1}"']
\arrow[rd, "f_{1}"]
& \\
A_{0}
& & B_{1} \\
& B_{0}
\arrow[from=lu, "f_{0}"']
\arrow[from=ru, "d_{1}"]
&
\end{tikzcd}
\end{equation*}
Note the identity functor is a discrete opfibration, 
and the composite of discrete opfibrations is a discrete opfibration,
by the Pullback Pasting Lemma. 
\end{example}

\begin{defn}
Let $\DOpf(\E)$ be the category whose objects are internal categories 
and whose morphisms are discrete opfibrations.
\end{defn}

%%%%%%%%%%%%%%%%%%%%%%%%%%%%%%%%%%%%%%%%%%%%%%%%%%%%%%%%%%%%%%%%%%%%%%%%
\section{Internal cofunctors}\label{S:internalcofun}
%%%%%%%%%%%%%%%%%%%%%%%%%%%%%%%%%%%%%%%%%%%%%%%%%%%%%%%%%%%%%%%%%%%%%%%%

This section introduces the notion of an internal cofunctor and proves a
useful representation of internal cofunctors as certain spans of 
internal functors. 
Since their introduction \cite{Agu97, HM93} there has been almost no 
work on cofunctors, apart from the recent reference \cite{AU17}.
To avoid confusion, we explicitly note that a cofunctor is \emph{not} a 
contravariant functor. 

The idea of a cofunctor is to generalise discrete opfibrations, 
providing a way to lift certain morphisms while preserving identities 
and composition.
Cofunctors are dual to functors in the sense that they \emph{lift} 
morphisms in the opposite direction to the object assignment, while 
functors \emph{push-forward} morphisms in the same direction. 
In the context of synchronisation, a cofunctor will later be interpreted
as the \textsf{Put} component of an internal lens which lifts anchored 
view updates in the pullback 
$\Lambda_{1} \coloneqq A_{0} \times_{B_{0}} B_{1}$ to source updates in
$A_{1}$.

\begin{defn}\label{defn:cofunctor}
Let $A$ and $B$ be internal categories. 
An \emph{internal cofunctor} $\varphi \colon B \nrightarrow A$ consists 
of morphisms, 
\[
	\varphi_{0} \colon A_{0} \longrightarrow B_{0}
	\qquad \qquad
	\varphi_{1} \colon \Lambda_{1} \longrightarrow A_{1}
	\qquad \qquad
	p_{0} \colon \Lambda_{1} \longrightarrow A_{0}
\]
together with the pullbacks,
\begin{equation}\label{dgrm:cofunctorpb}
\begin{tikzcd}[row sep = small, column sep = small]
& \Lambda_{1}
\arrow[ld, "d_{1}"']
\arrow[rd, "\phibar_{1}"]
& \\
A_{0}
& & B_{1} \\
& B_{0}
\arrow[from=lu, "\varphi_{0}"']
\arrow[from=ru, "d_{1}"]
\arrow[from=uu, phantom, "\lrcorner" rotate = -45, very near start]
&
\end{tikzcd}
\qquad \qquad
\begin{tikzcd}[row sep = small, column sep = small]
& \Lambda_{2}
\arrow[ld, "d_{2}"']
\arrow[rd, "\phibar_{2}"]
& \\
\Lambda_{1}
& & B_{2} \\
& B_{1}
\arrow[from=lu, "\phibar_{1}"']
\arrow[from=ru, "d_{2}"]
\arrow[from=uu, phantom, "\lrcorner" rotate = -45, very near start]
&
\end{tikzcd}
\end{equation}
such that the following diagrams commute:
\begin{equation}\label{dgrm:internalcofunctor}
\begin{tikzcd}
\Lambda_{1}
\arrow[r, "p_{0}"]
\arrow[d, "\phibar_{1}"']
& A_{0}
\arrow[d, "\varphi_{0}"]
\\
B_{1}
\arrow[r, "d_{0}"']
& B_{0}
\end{tikzcd}
\quad
\begin{tikzcd}
A_{0}
\arrow[d, "1_{A_{0}}"']
&
\Lambda_{1}
\arrow[l, "d_{1}"']
\arrow[r, "p_{0}"]
\arrow[d, "\varphi_{1}"]
& A_{0}
\arrow[d, "1_{A_{0}}"]
\\
A_{0}
& 
A_{1}
\arrow[l, "d_{1}"]
\arrow[r, "d_{0}"']
& A_{0}
\end{tikzcd}
\quad
\begin{tikzcd}
A_{0}
\arrow[r, "i_{0}"]
\arrow[d, "1_{A_{0}}"']
& \Lambda_{1}
\arrow[d, "\varphi_{1}"]
\\
A_{0}
\arrow[r, "i_{0}"']
& A_{1}
\end{tikzcd}
\quad
\begin{tikzcd}
\Lambda_{2}
\arrow[r, "d_{1}"]
\arrow[d, "\varphi_{2}"']
& \Lambda_{1}
\arrow[d, "\varphi_{1}"]
\\
A_{2}
\arrow[r, "d_{1}"']
& A_{1}
\end{tikzcd}
\end{equation}
\end{defn}

\begin{rem}
The pullback projections in \eqref{dgrm:cofunctorpb} will play different
roles which prompt different notational conventions. 
The projection $d_{1} \colon \Lambda_{1} \rightarrow A_{0}$ should be 
understood as the domain map for an internal category with object of 
morphisms $\Lambda_{1}$ which will be defined in 
Proposition~\ref{prop:cofunctor}.
The projection $\phibar_{1} \colon \Lambda_{1} \rightarrow B_{1}$ should
be understood as morphism assignment for a discrete opfibration 
$\phibar$ which will be defined in Theorem~\ref{thm:cofunctor}.
The projections $d_{2}$ and $\phibar_{2}$ for $\Lambda_{2}$ may be
understood similarly. 
\end{rem}

\begin{notation}
The commutative diagrams \eqref{dgrm:internalcofunctor} include 
morphisms defined using the universal property of the pullback via the 
diagrams below:
\begin{equation}\label{dgrm:cofunctoruni}
\begin{tikzcd}[row sep = small, column sep = small]
& A_{0}
\arrow[lddd, bend right, "1_{A_{0}}"']
\arrow[rd, "\varphi_{0}"]
\arrow[dd, dashed, "i_{0}"]
& \\
& & 
B_{0}
\arrow[dd, "i_{0}"]
\\[-2.5ex]
& \Lambda_{1}
\arrow[ld, "d_{1}"']
\arrow[rd, "\phibar_{1}"]
& \\
A_{0}
& & B_{1} \\
& B_{0}
\arrow[from=lu, "\varphi_{0}"']
\arrow[from=ru, "d_{1}"]
\arrow[from=uu, phantom, "\lrcorner" rotate = -45, very near start]
&
\end{tikzcd}
\quad
\begin{tikzcd}[row sep = small, column sep = small]
& \Lambda_{2}
\arrow[ld, "d_{2}"']
\arrow[rd, "\phibar_{2}"]
\arrow[dd, dashed, "d_{1}"]
& \\
\Lambda_{1}
\arrow[dd, "d_{1}"']
& & 
B_{2}
\arrow[dd, "d_{1}"]
\\[-2.5ex]
& \Lambda_{1}
\arrow[ld, "d_{1}"']
\arrow[rd, "\phibar_{1}"]
& \\
A_{0}
& & B_{1} \\
& B_{0}
\arrow[from=lu, "\varphi_{0}"']
\arrow[from=ru, "d_{1}"]
\arrow[from=uu, phantom, "\lrcorner" rotate = -45, very near start]
&
\end{tikzcd}
\quad
\begin{tikzcd}[row sep = small, column sep = small]
& \Lambda_{2}
\arrow[ld, "d_{2}"']
\arrow[rd, "\phibar_{2}"]
\arrow[dd, dashed, "p_{0}"]
& \\
\Lambda_{1}
\arrow[dd, "p_{0}"']
& & 
B_{2}
\arrow[dd, "d_{0}"]
\\[-2.5ex]
& \Lambda_{1}
\arrow[ld, "d_{1}"']
\arrow[rd, "\phibar_{1}"]
& \\
A_{0}
& & B_{1} \\
& B_{0}
\arrow[from=lu, "\varphi_{0}"']
\arrow[from=ru, "d_{1}"]
\arrow[from=uu, phantom, "\lrcorner" rotate = -45, very near start]
&
\end{tikzcd}
\quad
\begin{tikzcd}[row sep = small, column sep = small]
& \Lambda_{2}
\arrow[ld, "d_{2}"']
\arrow[rd, "p_{0}"]
\arrow[dd, dashed, "\varphi_{2}"]
& \\
\Lambda_{1}
\arrow[dd, "\varphi_{1}"']
& & 
\Lambda_{1}
\arrow[dd, "\varphi_{1}"]
\\[-2.5ex]
& A_{2}
\arrow[ld, "d_{2}"']
\arrow[rd, "d_{0}"]
& \\
A_{1}
& & A_{1} \\
& A_{0}
\arrow[from=lu, "d_{0}"']
\arrow[from=ru, "d_{1}"]
\arrow[from=uu, phantom, "\lrcorner" rotate = -45, very near start]
&
\end{tikzcd}
\end{equation}
\end{notation}

\begin{rem}
Strictly speaking, the morphism 
$p_{0} \colon \Lambda_{1} \rightarrow A_{0}$ is not required for the
definition of a cofunctor. 
Instead the two commutative diagrams in \eqref{dgrm:internalcofunctor}
which contain it may be replaced with the commutative diagram: 
\begin{equation}
\begin{tikzcd}
\Lambda_{1}
\arrow[d, "\phibar_{1}"']
\arrow[r, "\varphi_{1}"]
& A_{1}
\arrow[r, "d_{0}"]
& A_{0}
\arrow[d, "\varphi_{0}"]
\\
B_{1}
\arrow[rr, "d_{0}"']
& & B_{0}
\end{tikzcd}
\end{equation}
\end{rem}

\begin{example}
An internal cofunctor with 
$\varphi_{1} \colon \Lambda_{1} \cong A_{1}$ is a discrete opfibration.
\end{example}

\begin{example}
An internal cofunctor between monoids, as categories with one object, 
is a monoid homomorphism. 
\end{example}

\begin{example}
An internal cofunctor with $\varphi_{0} = 1_{A_{0}}$ is an 
identity-on-objects functor. 
\end{example}

\begin{rem}
Given an internal category $A$, the \emph{identity cofunctor} consists
of morphisms:
\[
	1_{A_{0}} \colon A_{0} \longrightarrow A_{0}
	\qquad
	1_{A_{1}} \colon A_{1} \longrightarrow A_{1}
	\qquad
	d_{0} \colon A_{1} \longrightarrow A_{0}
\]
Given internal cofunctors $\varphi \colon B \nrightarrow A$ and 
$\gamma \colon C \nrightarrow B$, consisting of triples 
$(\varphi_{0}, \varphi_{1}, p_{0})$ and 
$(\gamma_{0}, \gamma_{1}, q_{0})$ respectively, their 
\emph{composite cofunctor} 
$\varphi \circ \gamma \colon C \rightarrow A$
consists of the morphism, 
\[
	\gamma_{0} \varphi_{0} \colon A_{0} \longrightarrow C_{0} \\
\]
together with the pullback $A_{0} \times_{C_{0}} C_{1}$ and the
morphisms, 
\begin{equation}\label{eqn:cofcomp}
	\varphi_{1} \langle \pi_{0}, 
	\gamma_{1}(\varphi_{0} \times 1_{C_{1}})\rangle \colon 
	A_{0} \times_{C_{0}} C_{1} \longrightarrow A_{1} 
	\qquad
	p_{0} \langle \pi_{0}, 
	\gamma_{1}(\varphi_{0} \times 1_{C_{1}})\rangle \colon 
	A_{0} \times_{C_{0}} C_{1} \longrightarrow A_{0} 
\end{equation}
where the universal morphisms are defined via the following commutative
diagram:
\begin{equation}
\begin{tikzcd}[column sep = large]
A_{0} \times_{C_{0}} C_{1} 
\arrow[d, "{\langle \pi_{0}, \gamma_{1}(\varphi_{0} \times 1_{C_{1}})\rangle}"']
\arrow[r, "\varphi_{0} \times 1_{C_{1}}"]
\arrow[rd, phantom, "\lrcorner", pos=0.05]
& \Omega_{1}
\arrow[d, "\gamma_{1}"]
\arrow[r, "\overline{\gamma}_{1}"]
\arrow[ddd, bend right, "d_{1}"' pos = 0.64]
& C_{1}
\arrow[ddd, "d_{1}"]
\\
\Lambda_{1}
\arrow[r, crossing over, "\phibar_{1}"]
\arrow[d, "\varphi_{1}"]
\arrow[dd, bend right, "d_{1}"']
& B_{1}
\arrow[dd, "d_{1}"]
& \\
A_{1}
\arrow[d, "d_{1}"]
& & \\
A_{0}
\arrow[r, "\varphi_{0}"']
& B_{0}
\arrow[r, "\gamma_{0}"']
& C_{0}
\end{tikzcd}
\end{equation}
Composition of cofunctors is both unital and associative, however we 
omit the diagram-chasing required for the proof. 
\end{rem}

\begin{defn}
Let $\Cof(\mathcal{E})$ be the category whose objects are internal 
categories and whose morphisms are internal cofunctors. 
\end{defn}

\begin{prop}\label{prop:cofunctor}
If $\varphi \colon B \nrightarrow A$ is an internal cofunctor, 
then there exists an internal category $\Lambda$ with object of objects
$A_{0}$ and object of morphisms $\Lambda_{1}$, together with 
domain map $d_{1} \colon \Lambda_{1} \rightarrow A_{0}$, 
codomain map $p_{0} \colon \Lambda_{1} \rightarrow A_{0}$, 
identity map $i_{0} \colon A_{0} \rightarrow \Lambda_{1}$, 
and composition map $d_{1} \colon \Lambda_{2} \rightarrow \Lambda_{1}$. 
\end{prop}
\begin{proof}
We give a partial proof and show the first pair of diagrams in 
\eqref{dgrm:internalcat} are satisfied. 
Using the relevant diagrams from Definition~\ref{defn:internalcat} 
and Definition~\ref{defn:cofunctor} we have the following commutative 
diagram: 
\begin{equation*}
\begin{tikzcd}[row sep = small]
& & A_{0}
\arrow[lldd, "1_{A_{0}}"']
\arrow[dd, "i_{0}"']
\arrow[rd, "i_{0}"']
\arrow[rrdd, bend left, "1_{A_{0}}"]
& & \\
& & & A_{1}
\arrow[rd, "d_{0}"' pos = 0.4]
& \\
A_{0}
& & \Lambda_{1}
\arrow[ru, "\varphi_{1}"]
\arrow[ll, "d_{1}"]
\arrow[rr, "p_{0}"']
& & A_{0}
\end{tikzcd}
\end{equation*}
This shows that the identity map 
$i \colon A_{0} \rightarrow \Lambda_{1}$ is well-defined. 

To show that $\Lambda_{2}$ is well-defined as the the pullback of the 
domain and codomain maps (left-most square below) 
we use the Pullback Pasting Lemma, noting that the outer rectangles 
below are equal: 
\begin{equation*}
\begin{tikzcd}
\Lambda_{2} 
\arrow[d, "d_{2}"']
\arrow[r, "p_{0}"]
& \Lambda_{1}
\arrow[d, "d_{1}"']
\arrow[r, "\phibar_{1}"]
\arrow[rd, phantom, "\lrcorner" very near start]
& B_{1}
\arrow[d, "d_{1}"]
\\
\Lambda_{1}
\arrow[r, "p_{0}"']
& A_{0}
\arrow[r, "\varphi_{0}"']
& B_{0}
\end{tikzcd}
\qquad = \qquad
\begin{tikzcd}
\Lambda_{2} 
\arrow[d, "d_{2}"']
\arrow[r, "\phibar_{2}"]
\arrow[rd, phantom, "\lrcorner" very near start]
& B_{2}
\arrow[rd, phantom, "\lrcorner" very near start]
\arrow[r, "d_{0}"]
\arrow[d, "d_{2}"']
& B_{1}
\arrow[d, "d_{1}"]
\\
\Lambda_{1}
\arrow[r, "\phibar_{1}"']
& B_{1}
\arrow[r, "d_{0}"']
& B_{0}
\end{tikzcd}
\end{equation*}

Again using the relevant diagrams from Definition~\ref{defn:internalcat} 
and Definition~\ref{defn:cofunctor} we have the following commutative 
diagram:
\begin{equation*}
\begin{tikzcd}[row sep = small]
\Lambda_{1}
\arrow[ddd, "d_{1}"']
& & & 
\Lambda_{2}
\arrow[lll, "d_{2}"']
\arrow[rrr, "p_{0}"]
\arrow[ddd, "d_{1}"']
\arrow[rd, "\varphi_{2}"']
& & & 
\Lambda_{1}
\arrow[ld, "\varphi_{1}"]
\arrow[ddd, "p_{0}"]
\\
& & & & 
A_{2}
\arrow[d, "d_{1}"']
\arrow[r, "d_{0}"]
& A_{1}
\arrow[d, "d_{0}"]
& \\[+1.5ex]
& & & & 
A_{1}
\arrow[r, "d_{0}"']
& A_{0}
\arrow[rd, equal]
& \\
A_{0}
& & & 
\Lambda_{1}
\arrow[lll, "d_{1}"]
\arrow[rrr, "p_{0}"']
\arrow[ru, "\varphi_{1}"]
& & & 
A_{0}
\end{tikzcd}
\end{equation*}
This shows that the composition map 
$d_{1} \colon \Lambda_{2} \rightarrow \Lambda_{1}$ is well-defined.
\end{proof}

\begin{rem}
Proposition~\ref{prop:cofunctor} may be understood as showing that a
cofunctor induces a category whose objects are source states and whose 
morphisms are anchored view updates. 
The internal category $\Lambda$ is shown in Theorem~\ref{thm:cofunctor} 
to mediate between the source and the view, 
and reduces the complexity of Definition~\ref{defn:cofunctor}
to a simple statement concerning internal categories and functors. 
\end{rem}

\begin{theorem}\label{thm:cofunctor}
If $\varphi \colon B \nrightarrow A$ is an internal cofunctor, 
then there is an internal discrete opfibration 
$\phibar \colon \Lambda \rightarrow B$ consisting of the morphisms,
\[
	\varphi_{0} \colon A_{0} \longrightarrow B_{0}
	\qquad \qquad
	\phibar_{1} \colon \Lambda_{1} \longrightarrow B_{1}
\] 
and an identity-on-objects internal functor 
$\varphi \colon \Lambda \rightarrow A$ consisting of morphisms:
\[
	1_{A_{0}} \colon A_{0} \longrightarrow A_{0}
	\qquad \qquad
	\varphi_{1} \colon \Lambda_{1} \longrightarrow A_{1}
\]
Thus every internal cofunctor $\varphi \colon B \nrightarrow A$ may be 
represented as a span of internal functors:
\begin{equation*}
\begin{tikzcd}[column sep = small]
& \Lambda
\arrow[ld, "\phibar"']
\arrow[rd, "\varphi"]
& \\
B 
& & A
\end{tikzcd}
\end{equation*}
\end{theorem}
\begin{proof}
To show that $\phibar \colon \Lambda \rightarrow B$ is a well-defined 
internal discrete opfibration, we note from \eqref{dgrm:cofunctorpb}, 
\eqref{dgrm:internalcofunctor}, and \eqref{dgrm:cofunctoruni} that the 
following diagrams commute: 
\begin{equation*}
\begin{tikzcd}
A_{0}
\arrow[d, "\varphi_{0}"']
& \Lambda_{1}
\arrow[l, "d_{1}"']
\arrow[r, "p_{0}"]
\arrow[d, "\phibar_{1}"]
\arrow[ld, phantom, "\llcorner" very near start, rotate = -45]
& A_{0}
\arrow[d, "\varphi_{0}"]
\\
B_{0}
& B_{1}
\arrow[l, "d_{1}"]
\arrow[r, "d_{0}"']
& B_{0}
\end{tikzcd}
\qquad
\begin{tikzcd}
A_{0}
\arrow[d, "\varphi_{0}"']
\arrow[r, "i_{0}"]
& \Lambda_{1}
\arrow[d, "\phibar_{1}"]
\\
B_{0}
\arrow[r, "i_{0}"']
& B_{1}
\end{tikzcd}
\qquad
\begin{tikzcd}
\Lambda_{2}
\arrow[d, "\phibar_{2}"']
\arrow[r, "d_{1}"]
& \Lambda_{1}
\arrow[d, "\phibar_{1}"]
\\
B_{2}
\arrow[r, "d_{1}"']
& B_{1}
\end{tikzcd}
\end{equation*}
To show that $\varphi \colon \Lambda \rightarrow A$ is a well-defined 
identity-on-objects internal functor, we again note from 
\eqref{dgrm:cofunctorpb}, \eqref{dgrm:internalcofunctor}, 
and \eqref{dgrm:cofunctoruni} that the following diagrams commute: 
\begin{equation*}
\begin{tikzcd}
A_{0}
\arrow[d, "1_{A_{0}}"']
& \Lambda_{1}
\arrow[l, "d_{1}"']
\arrow[r, "p_{0}"]
\arrow[d, "\varphi_{1}"]
& A_{0}
\arrow[d, "1_{A_{0}}"]
\\
A_{0}
& A_{1}
\arrow[l, "d_{1}"]
\arrow[r, "d_{0}"']
& A_{0}
\end{tikzcd}
\qquad
\begin{tikzcd}
A_{0}
\arrow[d, "1_{A_{0}}"']
\arrow[r, "i_{0}"]
& \Lambda_{1}
\arrow[d, "\varphi_{1}"]
\\
A_{0}
\arrow[r, "i_{0}"']
& A_{1}
\end{tikzcd}
\qquad
\begin{tikzcd}
\Lambda_{2}
\arrow[d, "\varphi_{2}"']
\arrow[r, "d_{1}"]
& \Lambda_{1}
\arrow[d, "\varphi_{1}"]
\\
A_{2}
\arrow[r, "d_{1}"']
& A_{1}
\end{tikzcd}
\end{equation*}
Thus every internal cofunctor may be represented as a span of internal 
functors, with left-leg an internal discrete opfibration, and 
right-leg an identity-on-objects internal functor. 
\end{proof}

%%%%%%%%%%%%%%%%%%%%%%%%%%%%%%%%%%%%%%%%%%%%%%%%%%%%%%%%%%%%%%%%%%%%%%%%
\section{Internal Lenses}
%%%%%%%%%%%%%%%%%%%%%%%%%%%%%%%%%%%%%%%%%%%%%%%%%%%%%%%%%%%%%%%%%%%%%%%%

In this section we define an internal lens to consist of an 
internal \textsf{Get} functor and an internal \textsf{Put} cofunctor 
satisfying a simple axiom akin to the \textsf{Put-Get} law. 
An immediate corollary of Theorem~\ref{thm:cofunctor} is that every
internal lens may be understood as a particular commuting triangle 
\eqref{dgrm:lenstriangle} of internal functors. 
We also construct a category whose objects are internal categories and 
whose morphisms are internal lenses. 
The section concludes with a unification of discrete opfibrations, 
state-based lenses, c-lenses, and d-lenses in this internal framework,
based upon results in \cite{Cla18}.

\begin{defn}\label{defn:internallens}
An \emph{internal lens} $(f, \varphi) \colon A \br B$ 
consists of an internal functor $f \colon A \rightarrow B$ 
comprised of morphisms,
\[
	f_{0} \colon A_{0} \longrightarrow B_{0}
	\qquad \qquad
	f_{1} \colon A_{1} \longrightarrow B_{1}
\]
and an internal cofunctor $\varphi \colon B \nrightarrow A$ comprised 
of morphisms, 
\[
	\varphi_{0} \colon A_{0} \longrightarrow B_{0}
	\qquad \qquad
	\varphi_{1} \colon \Lambda_{1} \longrightarrow A_{1}
	\qquad \qquad
	p_{0} \colon \Lambda_{1} \longrightarrow A_{0}
\]
such that $\varphi_{0} = f_{0}$ and the following diagram commutes: 
\begin{equation}\label{dgrm:internallens}
\begin{tikzcd}[column sep = small]
& \Lambda_{1}
\arrow[ld, "\varphi_{1}"']
\arrow[rd, "1_{\Lambda_{1}}"]
& \\
A_{1}
\arrow[rr, "{\langle d_{1}, f_{1} \rangle}"']
& & \Lambda_{1}
\end{tikzcd}
\end{equation}
\end{defn}

\begin{rem}
Alternatively, the commutative diagram \eqref{dgrm:internallens} for an 
internal lens may be 
replaced with the requirement that the following diagram commutes: 
\begin{equation}
\begin{tikzcd}[column sep = small]
& \Lambda_{1}
\arrow[ld, "\varphi_{1}"']
\arrow[rd, "\phibar_{1}"]
& \\
A_{1}
\arrow[rr, "f_{1}"']
& & B_{1}
\end{tikzcd}
\end{equation}
In either case, this axiom for an internal lens ensures that the functor
and cofunctor parts interact as expected. 
Explicitly it states that \emph{lifting} a morphism by the cofunctor 
then \emph{pushing-forward} by the functor should return the original 
morphism.
\end{rem}

\begin{cor}
Every internal lens $(f, \varphi) \colon A \br B$ may be represented as 
a commuting triangle of internal functors,
\begin{equation}\label{dgrm:lenstriangle}
\begin{tikzcd}[column sep = small]
& \Lambda
\arrow[ld, "\varphi"']
\arrow[rd, "\phibar"]
& \\
A 
\arrow[rr, "f"']
& & B
\end{tikzcd}
\end{equation}
where $\phibar \colon \Lambda \rightarrow B$ is an internal discrete 
opfibration, and $\varphi \colon \Lambda \rightarrow A$ is an 
identity-on-objects internal functor. 
\end{cor}

\begin{cor}
Given a pair of internal lenses $(f, \varphi) \colon A \br B$ and 
$(g, \gamma) \colon B \br C$, their composite internal lens may be 
computed via the composition of the respective functor and cofunctor 
parts, and has a simple representation using the pullback of internal 
functors:
\begin{equation}
\begin{tikzcd}[row sep = small, column sep = small]
&[+1em] & \Lambda \times_{B} \Omega 
\arrow[ld]
\arrow[rd]
\arrow[dd, phantom, very near start, "\lrcorner" rotate = -45]
& &[+1em] \\
& \Lambda 
\arrow[ld, "\varphi"']
\arrow[rd, "\phibar"]
& & \Omega
\arrow[ld, "\gamma"']
\arrow[rd, "\overline{\gamma}"]
& \\
A
\arrow[rr, "f"']
& & B
\arrow[rr, "g"']
& & C
\end{tikzcd}
\end{equation}
\end{cor}

\begin{defn}
Let $\Lens(\E)$ be the category whose objects are internal 
categories and whose morphisms are internal lenses.
Composition of internal lenses is determined by composition of the 
corresponding functor and cofunctor parts. 
\end{defn}

\begin{example}
Every discrete opfibration is both an internal functor and an internal 
cofunctor, hence also an internal lens. 
Therefore $\DOpf(\E)$ is a wide subcategory of $\Lens(\E)$. 
\end{example}

\begin{example}
If $\mathcal{E} = \Set$, then the category $\Lens(\Set)$ is the category
of d-lenses \cite{DXC11}. 
The \textsf{Get} of a d-lens $A \br B$ is given by a
functor $f \colon A \rightarrow B$, while the \textsf{Put} of a d-lens 
is given by a cofunctor  $\varphi \colon B \nrightarrow A$. 

In particular, the function 
$\varphi_{1} \colon \Lambda_{1} \rightarrow A_{1}$ takes 
each pair $(a, u \colon fa \rightarrow b) \in \Lambda_{1}$ to a morphism
$\varphi(a, u) \colon a \rightarrow p(a, u) \in A$, as illustrated in 
the diagram below. 
\begin{equation}
	\begin{tikzcd}
	A
	\arrow[d, harpoon', shift right, "f"']
	\arrow[d, leftharpoonup, shift left, "\varphi"]
	&[+1em] a
	\arrow[d, phantom, "\vdots"]
	\arrow[r, "{\varphi(a, u)}"]
	& p(a, u)
	\arrow[d, phantom, "\vdots"]
	\\
	B
	& fa
	\arrow[r, "u"]
	& b
	\end{tikzcd}
\end{equation}
The \textsf{Put-Get} law is satisfied by \eqref{dgrm:internallens},
which corresponds in the above diagram to the morphism $\varphi(a, u)$ 
being a genuine lift of $u \colon fa \rightarrow b$ with respect to the
functor acting on morphisms. 
The \textsf{Get-Put} and \textsf{Put-Put} laws are satisfied as 
$\varphi \colon \Lambda \rightarrow A$ is a functor, 
which respects identities and composition by definition. 
\end{example}

\begin{example} 
Every state-based lens (see \cite{FGMPS07}) consisting of 
\textsf{Get} function $f \colon A \rightarrow B$ and \textsf{Put} 
function $p \colon A \times B \rightarrow A$ induces a lens in 
$\Lens(\Set)$. 

Let $\widehat{A}$ and $\widehat{B}$ be the small codiscrete categories 
induced by the sets $A$ and $B$, respectively, and let 
$f \colon \widehat{A} \rightarrow \widehat{B}$ be the canonical functor,
\begin{equation*}
\begin{tikzcd}
A 
\arrow[d, "f"']
& A \times A
\arrow[l, "\pi_{0}"']
\arrow[r, "\pi_{1}"]
\arrow[d, "f \times f"]
& A 
\arrow[d, "f"]
\\
B 
& B \times B
\arrow[l, "\pi_{0}"]
\arrow[r, "\pi_{1}"']
& B
\end{tikzcd}
\end{equation*}
induced by the \textsf{Get} function. 
Let $\Lambda$ be the category with domain and codomain maps described 
by the span: 
\begin{equation*}
\begin{tikzcd}[row sep = small, column sep = tiny]
& A \times B
\arrow[ld, "\pi_{0}"']
\arrow[rd, "p"]
& \\
A
& & A
\end{tikzcd}
\end{equation*}
The category $\Lambda$ is well-defined by the lens laws. 
The functor $\phibar \colon \Lambda \rightarrow \widehat{B}$ is induced
using the \textsf{Put-Get} law, 
\begin{equation*}
\begin{tikzcd}
A 
\arrow[d, "f"']
& A \times B
\arrow[ld, phantom, "\llcorner" very near start, rotate = -45]
\arrow[l, "\pi_{0}"']
\arrow[r, "p"]
\arrow[d, "f \times 1_{B}"]
& A 
\arrow[d, "f"]
\\
B 
& B \times B
\arrow[l, "\pi_{0}"]
\arrow[r, "\pi_{1}"']
& B
\end{tikzcd}
\end{equation*}
while the functor $\varphi \colon \Lambda \rightarrow \widehat{A}$ is
induced for free: 
\begin{equation*}
\begin{tikzcd}
A 
\arrow[d, "1_{A}"']
& A \times B
\arrow[l, "\pi_{0}"']
\arrow[r, "p"]
\arrow[d, "{\langle \pi_{0}, p \rangle}"]
& A 
\arrow[d, "1_{A}"]
\\
A 
& A \times A
\arrow[l, "\pi_{0}"]
\arrow[r, "\pi_{1}"']
& A
\end{tikzcd}
\end{equation*}
This example may be instantiated internal to any category $\E$ with 
finite limits. 
\end{example}

\begin{example}
Given a pair of state-based lenses $(f, p) \colon A \br B$ and 
$(g, q) \colon B \br C$, 
their composite is a lens whose \textsf{Get} function is given 
by $gf \colon A \rightarrow C$ and whose \textsf{Put} function may be 
computed from the formula \eqref{eqn:cofcomp}:
\[
	p\langle \pi_{0}, q(f \times 1_{C})\rangle 
	\colon A \times C \longrightarrow A	
\]
\end{example}

\begin{example}
Every c-lens (also known as a split opfibration, see \cite{JRW12}) 
consisting of a \textsf{Get} functor $f \colon A \rightarrow B$ and 
\textsf{Put} functor $p \colon \comma{f}{B} \rightarrow A$ induces a 
lens in $\Lens(\Cat)$. 

Let $\mathbb{B}$ be the double category of squares, whose category of 
objects is $B$ and whose category of morphisms is the arrow category 
$\Arr{B}$, together with domain and codomain functors 
$l, r \colon \Arr{B} \rightarrow B$ given by,
\begin{equation*}
\begin{tikzcd}
B_{1}
\arrow[d, "d_{1}"']
& B_{11}
\arrow[l, "d_{2}\pi_{0}"']
\arrow[r, "d_{0}\pi_{1}"]
\arrow[d, "d_{2}\pi_{1}"]
& B_{1}
\arrow[d, "d_{1}"]
\\
B_{0}
& B_{1}
\arrow[l, "d_{1}"]
\arrow[r, "d_{0}"']
& B_{0}
\end{tikzcd}
\qquad \qquad
\begin{tikzcd}
B_{1}
\arrow[d, "d_{0}"']
& B_{11}
\arrow[l, "d_{2}\pi_{0}"']
\arrow[r, "d_{0}\pi_{1}"]
\arrow[d, "d_{0}\pi_{0}"]
& B_{1}
\arrow[d, "d_{0}"]
\\
B_{0}
& B_{1}
\arrow[l, "d_{1}"]
\arrow[r, "d_{0}"']
& B_{0}
\end{tikzcd}
\end{equation*}
using the same notation from the diagram in Example~\ref{ex:arrowcat};
define $\mathbb{A}$ similarly. 
Construct the functor $\Arr{f} \colon \Arr{A} \rightarrow \Arr{B}$ 
between the arrow categories,
\begin{equation*}
\begin{tikzcd}
A_{1}
\arrow[d, "f_{1}"']
& A_{11}
\arrow[l, "d_{2}\pi_{0}"']
\arrow[r, "d_{0}\pi_{1}"]
\arrow[d, "f_{2} \times f_{2}"]
& A_{1}
\arrow[d, "f_{1}"]
\\
B_{1}
& B_{11}
\arrow[l, "d_{2}\pi_{0}"]
\arrow[r, "d_{0}\pi_{1}"']
& B_{1}
\end{tikzcd}
\end{equation*}
induced by the \textsf{Get} functor, 
which forms a canonical double functor 
$f \colon \mathbb{A} \rightarrow \mathbb{B}$.
 
Let $\mathbb{\Lambda}$ be the double category with domain and codomain 
functors described by the span: 
\begin{equation*}
\begin{tikzcd}[row sep = small, column sep = tiny]
& \comma{f}{B}
\arrow[ld, "l"']
\arrow[rd, "p"]
& \\
A
& & A
\end{tikzcd}
\end{equation*}
Note that the comma category $\comma{f}{B}$ may defined as the pullback,
\begin{equation*}
\begin{tikzcd}
\comma{f}{B}
\arrow[d, "l"']
\arrow[r]
\arrow[rr, bend left, "r"]
& \Arr{B} 
\arrow[d, "l"]
\arrow[r, "r"'] 
& B \\
A
\arrow[r, "f"']
& B
\arrow[from=lu, phantom, "\lrcorner" very near start]
&
\end{tikzcd}
\end{equation*}
where $l \colon \comma{f}{B} \rightarrow A$ and 
$r \colon \comma{f}{B} \rightarrow B$ are the usual comma category 
projections.
The double category $\mathbb{\Lambda}$ is well-defined by the c-lens 
laws, and we may show with further reasoning that there exist unique 
double functors 
$\varphi \colon \mathbb{\Lambda} \rightarrow \mathbb{A}$ and 
$\phibar \colon \mathbb{\Lambda} \rightarrow \mathbb{B}$. 
\end{example}

%%%%%%%%%%%%%%%%%%%%%%%%%%%%%%%%%%%%%%%%%%%%%%%%%%%%%%%%%%%%%%%%%%%%%%%%
\section{Conclusion and Future Work}\label{S:conclusion}
%%%%%%%%%%%%%%%%%%%%%%%%%%%%%%%%%%%%%%%%%%%%%%%%%%%%%%%%%%%%%%%%%%%%%%%%

In this paper it was shown that lenses may be defined internal to any 
category $\E$ with pullbacks, providing a significantly generalised yet
minimal framework to understand the notion of synchronisation between
systems. 
It was demonstrated that the enigmatic \textsf{Put} of a lens may be 
understood as a cofunctor, which has a simple description as a span of a
discrete opfibration and an identity-on-objects functor. 
The surprising characterisation of a lens as a functor/cofunctor pair 
both promotes the prevailing attitude of lenses as morphisms between 
categories, and yields a straightforward definition for composition in 
the category $\Lens(\E)$, which fits within a diagram of 
forgetful functors. 
\begin{equation*}
\begin{tikzcd}[row sep = tiny]
& & \Cof(\E)^{\textsf{op}}
\arrow[rd] & \\
\DOpf(\E) 
\arrow[r]
& \Lens(\E)
\arrow[ru]
\arrow[rd]
& & \E \\
& & \Kat(\E) 
\arrow[ru]
& 
\end{tikzcd}
\end{equation*}

The success of internal lenses in unifying the known examples of 
state-based lenses, c-lenses, and d-lenses promotes the effectiveness of
this perspective for use in applications such programming, databases,
and Model-Driven Engineering, and also anticipates many future
mathematical developments. 
Current work in progress indicates that $\Lens(\E)$ may be enhanced to a 
$2$-category through incorporating natural transformations between 
lenses, while consideration of spans in $\Lens(\E)$ leads towards a 
clarified understanding of symmetric lenses; both ideas which have been 
shown to be important in applications and the literature 
\cite{Dis17, HPW11}.  
In future work we will investigate examples of lenses internal to a
diverse range of categories, as well as taking steps towards a theory of
lenses between enriched categories.

\subsection*{Acknowledgements}
The author is grateful to Michael Johnson and the anonymous reviewers
for providing helpful feedback on this work. 
The author would also like to thank the organisers of the ACT2019 
conference. 

%%%%%%%%%%%%%%%%%%%%%%%%%%%%%%%%%%%%%%%%%%%%%%%%%%%%%%%%%%%%%%%%%%%%%%%%
% References 
%%%%%%%%%%%%%%%%%%%%%%%%%%%%%%%%%%%%%%%%%%%%%%%%%%%%%%%%%%%%%%%%%%%%%%%%

\bibliographystyle{eptcs}

\begin{thebibliography}{10}
\providecommand{\bibitemdeclare}[2]{}
\providecommand{\surnamestart}{}
\providecommand{\surnameend}{}
\providecommand{\urlprefix}{Available at }
\providecommand{\url}[1]{\texttt{#1}}
\providecommand{\href}[2]{\texttt{#2}}
\providecommand{\urlalt}[2]{\href{#1}{#2}}
\providecommand{\doi}[1]{doi:\urlalt{http://dx.doi.org/#1}{#1}}
\providecommand{\bibinfo}[2]{#2}

\bibitemdeclare{phdthesis}{Agu97}
\bibitem{Agu97}
\bibinfo{author}{Marcelo \surnamestart Aguiar\surnameend}
  (\bibinfo{year}{1997}): \emph{\bibinfo{title}{Internal Categories and Quantum
  Groups}}.
\newblock Ph.D. thesis, \bibinfo{school}{Cornell University}.
\newblock \urlprefix\url{http://pi.math.cornell.edu/~maguiar/thesis2.pdf}.

\bibitemdeclare{inproceedings}{AU17}
\bibitem{AU17}
\bibinfo{author}{Danel \surnamestart Ahman\surnameend} \&
  \bibinfo{author}{Tarmo \surnamestart Uustalu\surnameend}
  (\bibinfo{year}{2017}): \emph{\bibinfo{title}{Taking Updates Seriously}}.
\newblock In: {\sl \bibinfo{booktitle}{Proceedings of the 6\textsuperscript{th}
  International Workshop on Bidirectional Transformations}}, {\sl
  \bibinfo{series}{CEUR Workshop Proceedings}} \bibinfo{volume}{1827}, pp.
  \bibinfo{pages}{59--73}.
\newblock \urlprefix\url{http://ceur-ws.org/Vol-1827/paper11.pdf}.

\bibitemdeclare{article}{BS81}
\bibitem{BS81}
\bibinfo{author}{F.~\surnamestart Bancilhon\surnameend} \&
  \bibinfo{author}{N.~\surnamestart Spyratos\surnameend}
  (\bibinfo{year}{1981}): \emph{\bibinfo{title}{Update Semantics of Relational
  Views}}.
\newblock {\sl \bibinfo{journal}{ACM Transactions on Database Systems}}
  \bibinfo{volume}{6}(\bibinfo{number}{4}), pp. \bibinfo{pages}{557--575},
  \doi{10.1145/319628.319634}.

\bibitemdeclare{book}{Bor94}
\bibitem{Bor94}
\bibinfo{author}{Francis \surnamestart Borceux\surnameend}
  (\bibinfo{year}{1994}): \emph{\bibinfo{title}{Handbook of Categorical
  Algebra}}.
\newblock {\sl \bibinfo{series}{Encyclopedia of Mathematics and its
  Applications}}, \bibinfo{publisher}{Cambridge University
  Press}, \bibinfo{address}{Cambridge}, \doi{10.1017/CBO9780511525858}.

\bibitemdeclare{mastersthesis}{Cla18}
\bibitem{Cla18}
\bibinfo{author}{Bryce \surnamestart Clarke\surnameend} (\bibinfo{year}{2018}):
  \emph{\bibinfo{title}{Characterising Asymmetric Lenses using Internal
  Categories}}.
\newblock Master's thesis, \bibinfo{school}{Macquarie University}.
\newblock \urlprefix\url{http://hdl.handle.net/1959.14/1268984}.

\bibitemdeclare{inproceedings}{Dis17}
\bibitem{Dis17}
\bibinfo{author}{Zinovy \surnamestart Diskin\surnameend}
  (\bibinfo{year}{2017}): \emph{\bibinfo{title}{Compositionality of Update
  Propagation: Laxed PutPut}}.
\newblock In: {\sl \bibinfo{booktitle}{Proceedings of the 6\textsuperscript{th}
  International Workshop on Bidirectional Transformations}}, {\sl
  \bibinfo{series}{CEUR Workshop Proceedings}} \bibinfo{volume}{1827}, pp.
  \bibinfo{pages}{74--89}.
\newblock \urlprefix\url{http://ceur-ws.org/Vol-1827/paper12.pdf}.

\bibitemdeclare{article}{DXC11}
\bibitem{DXC11}
\bibinfo{author}{Zinovy \surnamestart Diskin\surnameend},
  \bibinfo{author}{Yingfei \surnamestart Xiong\surnameend} \&
  \bibinfo{author}{Krzysztof \surnamestart Czarnecki\surnameend}
  (\bibinfo{year}{2011}): \emph{\bibinfo{title}{From State- to Delta-Based
  Bidirectional Model Transformations: the Asymmetric Case}}.
\newblock {\sl \bibinfo{journal}{Journal of Object Technology}}
  \bibinfo{volume}{10}, pp. \bibinfo{pages}{6:1--25},
  \doi{10.5381/jot.2011.10.1.a6}.

\bibitemdeclare{article}{FGMPS07}
\bibitem{FGMPS07}
\bibinfo{author}{J.~Nathan \surnamestart Foster\surnameend},
  \bibinfo{author}{Michael~B. \surnamestart Greenwald\surnameend},
  \bibinfo{author}{Jonathan~T. \surnamestart Moore\surnameend},
  \bibinfo{author}{Benjamin~C. \surnamestart Pierce\surnameend} \&
  \bibinfo{author}{Alan \surnamestart Schmitt\surnameend}
  (\bibinfo{year}{2007}): \emph{\bibinfo{title}{Combinators for Bidirectional
  Tree Transformations: A Linguistic Approach to the View-Update Problem}}.
\newblock {\sl \bibinfo{journal}{ACM Transactions on Programming Languages and
  Systems}} \bibinfo{volume}{29}(\bibinfo{number}{3}),
  \doi{10.1145/1232420.1232424}.

\bibitemdeclare{inproceedings}{GJ12}
\bibitem{GJ12}
\bibinfo{author}{Jeremy \surnamestart Gibbons\surnameend} \&
  \bibinfo{author}{Michael \surnamestart Johnson\surnameend}
  (\bibinfo{year}{2012}): \emph{\bibinfo{title}{Relating Algebraic and
  Coalgebraic Descriptions of Lenses}}.
\newblock In: {\sl \bibinfo{booktitle}{Proceedings of the 1\textsuperscript{st}
  International Workshop on Bidirectional Transformations}}, {\sl
  \bibinfo{series}{Electronic Communications of the
  EASST}}~\bibinfo{volume}{49}, \doi{10.14279/tuj.eceasst.49.726}.

\bibitemdeclare{article}{HM93}
\bibitem{HM93}
\bibinfo{author}{Philip~J. \surnamestart Higgins\surnameend} \&
  \bibinfo{author}{Kirill C.~H. \surnamestart Mackenzie\surnameend}
  (\bibinfo{year}{1993}): \emph{\bibinfo{title}{Duality for base-changing
  morphisms of vector bundles, modules, {L}ie algebroids and {P}oisson
  structures}}.
\newblock {\sl \bibinfo{journal}{Mathematical Proceedings of the Cambridge
  Philosophical Society}} \bibinfo{volume}{114}(\bibinfo{number}{3}), pp.
  \bibinfo{pages}{471--488}, \doi{10.1017/S0305004100071760}.

\bibitemdeclare{article}{HPW11}
\bibitem{HPW11}
\bibinfo{author}{Martin \surnamestart Hofmann\surnameend},
  \bibinfo{author}{Benjamin \surnamestart Pierce\surnameend} \&
  \bibinfo{author}{Daniel \surnamestart Wagner\surnameend}
  (\bibinfo{year}{2011}): \emph{\bibinfo{title}{Symmetric Lenses}}.
\newblock {\sl \bibinfo{journal}{SIGPLAN Not.}}
  \bibinfo{volume}{46}(\bibinfo{number}{1}), \doi{10.1145/1925844.1926428}.

\bibitemdeclare{article}{JR13}
\bibitem{JR13}
\bibinfo{author}{Michael \surnamestart Johnson\surnameend} \&
  \bibinfo{author}{Robert \surnamestart Rosebrugh\surnameend}
  (\bibinfo{year}{2013}): \emph{\bibinfo{title}{Delta Lenses and
  Opfibrations}}.
\newblock {\sl \bibinfo{journal}{Electronic Communications of the EASST}}
  \bibinfo{volume}{57}, \doi{10.14279/tuj.eceasst.57.875}.

\bibitemdeclare{inproceedings}{JR16}
\bibitem{JR16}
\bibinfo{author}{Michael \surnamestart Johnson\surnameend} \&
  \bibinfo{author}{Robert \surnamestart Rosebrugh\surnameend}
  (\bibinfo{year}{2016}): \emph{\bibinfo{title}{Unifying Set-Based, Delta-Based
  and Edit-Based Lenses}}.
\newblock In: {\sl \bibinfo{booktitle}{Proceedings of the 5\textsuperscript{th}
  International Workshop on Bidirectional Transformations}}, {\sl
  \bibinfo{series}{CEUR Workshop Proceedings}} \bibinfo{volume}{1571}, pp.
  \bibinfo{pages}{1--13}.
\newblock \urlprefix\url{http://ceur-ws.org/Vol-1571/paper_13.pdf}.

\bibitemdeclare{article}{JRW10}
\bibitem{JRW10}
\bibinfo{author}{Michael \surnamestart Johnson\surnameend},
  \bibinfo{author}{Robert \surnamestart Rosebrugh\surnameend} \&
  \bibinfo{author}{Richard \surnamestart Wood\surnameend}
  (\bibinfo{year}{2010}): \emph{\bibinfo{title}{Algebras and Update
  Strategies}}.
\newblock {\sl \bibinfo{journal}{Journal of Universal Computer Science}}
  \bibinfo{volume}{16}(\bibinfo{number}{5}), pp. \bibinfo{pages}{729--748},
  \doi{10.3217/jucs-016-05-0729}.

\bibitemdeclare{article}{JRW12}
\bibitem{JRW12}
\bibinfo{author}{Michael \surnamestart Johnson\surnameend},
  \bibinfo{author}{Robert \surnamestart Rosebrugh\surnameend} \&
  \bibinfo{author}{Richard \surnamestart Wood\surnameend}
  (\bibinfo{year}{2012}): \emph{\bibinfo{title}{Lenses, Fibrations and
  Universal Translations}}.
\newblock {\sl \bibinfo{journal}{Mathematical Structures in Computer Science}}
  \bibinfo{volume}{22}(\bibinfo{number}{1}), pp. \bibinfo{pages}{25--42},
  \doi{10.1017/S0960129511000442}.

\bibitemdeclare{book}{Joh02}
\bibitem{Joh02}
\bibinfo{author}{Peter~T \surnamestart Johnstone\surnameend}
  (\bibinfo{year}{2002}): \emph{\bibinfo{title}{Sketches of an Elephant: A
  Topos Theory Compendium}}.
\newblock \bibinfo{publisher}{Oxford University Press}.

\bibitemdeclare{book}{Mac98}
\bibitem{Mac98}
\bibinfo{author}{Saunders \surnamestart Mac~Lane\surnameend}
  (\bibinfo{year}{1998}): \emph{\bibinfo{title}{Categories for the Working
  Mathematician}}, \bibinfo{edition}{2nd} edition.
\newblock {\sl \bibinfo{series}{Graduate Texts in
  Mathematics}}~\bibinfo{volume}{5}, \bibinfo{publisher}{Springer-Verlag},
  \bibinfo{address}{New York}, \doi{10.1007/978-1-4757-4721-8}.

\bibitemdeclare{inproceedings}{Str74}
\bibitem{Str74}
\bibinfo{author}{Ross \surnamestart Street\surnameend} (\bibinfo{year}{1974}):
  \emph{\bibinfo{title}{Fibrations and {Y}oneda's lemma in a 2-category}}.
\newblock In: {\sl \bibinfo{booktitle}{Category Seminar}},
  \bibinfo{series}{Lecture Notes in Mathematics 420},
  \bibinfo{publisher}{Springer}, pp. \bibinfo{pages}{104--133},
  \doi{10.1007/BFb0063102}.

\end{thebibliography}

\end{document}